\documentclass[11pt]{amsart}

\usepackage{color,url}
\usepackage{graphics,graphicx}

\newtheorem{theorem}{Theorem}[section]

\theoremstyle{definition}
\newtheorem{definition}[theorem]{Definition}
\newtheorem{example}[theorem]{Example}

\newtheorem{corollary}[theorem]{Corollary}

\theoremstyle{remark}

\usepackage{a4wide}

\begin{document}

\title{On the $r$-derangements of type B}

\begin{abstract}
Extensions of a set partition obtained by imposing bounds on the size of the parts and the coloring 
of some of the elements are examined. Combinatorial properties and the generating functions 
of some counting sequences associated 
with these partitions are established.  Connections with Riordan arrays are presented.
\end{abstract}

\author{Istv\'{a}n Mezo}
\address{School of Mathematics and Statistics, Nanjing University of Information Science and
Technology, Nanjing, 210044, P.~R.~China}
\email{istvanmezo81@gmailcom}

\author{Victor H. Moll}
\address{Department of Mathematics,
Tulane University, New Orleans, LA 70118}
\email{vhm@tulane.edu}

\author{Jos\'{e}~L.~{R}am\'{i}rez}
\address{Departamento de Matem\'{a}ticas, Universidad Nacional de Colombia, Bogot\'{a},
Colombia}
\email{jlramirezr@unal.edu.co}

\author{Diego Villamizar}
\address{Department of Mathematics,
Tulane University, New Orleans, LA 70118}
\email{dvillami@tulane.edu}


\date{\today}


\maketitle

\newcommand{\nn}{\nonumber}
\newcommand{\ba}{\begin{eqnarray}}
\newcommand{\ea}{\end{eqnarray}}
\newcommand{\dz}{\frac{d}{dz}}
\newcommand{\E}{{\mathfrak{E}}}
\newcommand{\F}{{\mathfrak{F}}}
\newcommand{\Ro}{{\mathfrak{R}}}
\newcommand{\ift}{\int_{0}^{\infty}}
\newcommand{\ifft}{\int_{- \infty}^{\infty}}
\newcommand{\no}{\noindent}
\newcommand{\X}{{\mathbb{X}}}
\newcommand{\Q}{{\mathbb{Q}}}
\newcommand{\R}{{\mathbb{R}}}
\newcommand{\Y}{{\mathbb{Y}}}
\newcommand{\Ftwo}{{{_{2}F_{1}}}}
\newcommand{\realpart}{\mathop{\rm Re}\nolimits}
\newcommand{\imagpart}{\mathop{\rm Im}\nolimits}
\newcommand{\C}{{\mathcal C}}
\newcommand{\T}{{\mathcal T}}
\newcommand{\Tc}{{\mathcal{\overline{T}}}}
\newcommand{\D}{{\mathcal D}}
\newcommand{\A}{{\mathcal A}}

\newtheorem{Definition}{\bf Definition}[section]
\newtheorem{Thm}[Definition]{\bf Theorem}
\newtheorem{Example}[Definition]{\bf Example}
\newtheorem{Lem}[Definition]{\bf Lemma}
\newtheorem{Note}[Definition]{\bf Note}
\newtheorem{Cor}[Definition]{\bf Corollary}
\newtheorem{Prop}[Definition]{\bf Proposition}
\newtheorem{Problem}[Definition]{\bf Problem}
\numberwithin{equation}{section}

\section{Introduction} \label{intro}
\setcounter{equation}{0}

Consider  two sets of $n$ symbols $[n] = \{ 1, \, 2, \ldots, n \}$ and  $[\overline{n}] = \{ \overline{1},
\, \overline{2}, \ldots, \overline{n} \}$, with $i \neq \overline{j}$, for any $i, \, j$. Define
$X_{n} = [n] \cup [ \overline{n}]$.  The symbol $\overline{j}$ is
called the \texttt{colored version} of the symbol $j$.  Naturally there are $(2n)!$ permutations of
$X_{n}$. Some of these  permutations respect the sign, that is, 
satisfy $\sigma(\overline{j}) = \overline{\sigma(j)}$. These are called \texttt{signed permutations} or
\texttt{permutations of type} $B$. General information about them and their relations to 
Coxeter groups appears in Section $8.1$ of \cite{bjorner-2005a}.

The number of signed permutations on $ [n]$ is $2^{n}n!$, since each one of them is formed by a 
permutation of $[n]$ and a choice of sign. An example is 
\begin{equation}
\label{pi-def}
\pi = \begin{pmatrix} 1 & 2 & 3 & 4 & \overline{1} & \overline{2} & \overline{3} & \overline{4} \\
\overline{2} & 1 & 3 & \overline{4} & 2 & \overline{1} & \overline{3} & 4
\end{pmatrix}.
\end{equation}
\noindent
Observe that the complete permutation is determined by the values of
$\pi(1), \, \pi(2), \, \pi(3), \, \pi(4)$ and these must be a permutation of $[4]$ with some choices of
overlines. The remaining images are determined from the sign rule.  In order to simplify notation,
only the first half of the bottom row in \eqref{pi-def} is retained and $\pi$ is now written simply (in the
so-called \texttt{line notation}) as
\begin{equation}
\pi = \overline{2} \, 1 \, 3 \, \overline{4}.
\end{equation}

\noindent
As in the classical case, it is possible to express a signed permutation as a product of disjoint
cycles \cite{brenti-1994b}. The notation for the cycles is explained with the example.
 $\pi = \overline{4}6\overline{3}51\overline{2} \, \overline{9}87$. In order to compute the cycle of $\pi$
containing $1$, start by ignoring the coloring to produce the cycle $1 \rightarrow 4 \rightarrow 5$. Now 
insert the color back as they appear in the one-line notation for $\pi$. This produces the cycle
$(1 \overline{4}5)$. Continuing this process gives the final expression for $\pi$ as
 $(1 \overline{4}5)(\overline{2}6)(\overline{3})(7 \overline{9})(8)$. A second example  illustrates a
 point that \texttt{could lead to confusion}. The permutations considered here have
 no fixed points, for instance 
 \begin{equation}
 \pi = \begin{pmatrix}
 1 & 2 & 3 & \overline{1} & \overline{2} & \overline{3} \\
 2 & 1 & \overline{3} & \overline{2} & \overline{1} & 3
 \end{pmatrix}.
 \end{equation}
 \noindent
 The short hand notation for $\pi$ is $1 \, 2 \, \overline{3}$ and its cycle decomposition, as explained
 above,  is written as $(12)(\overline{3})$. The interpretation of the last term, $(\overline{3})$, is not that
 $\overline{3}$ is a fixed point of $\pi$, but that $\pi(3) = \overline{3}$ and (necessarily)
 $\pi(\overline{3}) = 3$.

 The goal of the present work is to study a variety of functions for signed permutations
  in terms of their cycle structure.

Recall that ${ n \brack k }$, the (unsigned) Stirling number of the first kind, counts the number of permutations
of $n$ elements with $k$ disjoint cycles. The recurrence for ${ n \brack k} $ is 
\begin{equation}
\begin{bmatrix} n \\ k \end{bmatrix}  =
\begin{bmatrix} n-1 \\ k -1 \end{bmatrix} +
(n-1)  \begin{bmatrix} n-1 \\ k \end{bmatrix}, \quad \text{ for } n>k> 1,
\end{equation}
\noindent
with the initial/boundary conditions $\begin{bmatrix} n\\ n \end{bmatrix}  = 1$ and
$ \begin{bmatrix} n \\ 1 \end{bmatrix} = (n-1)!$ for $n>0$ and naturally
$\begin{bmatrix} n \\ k  \end{bmatrix} = 0$ for $n<k$.  Other 
functions include some restrictions on the
length of the cycles. For example, if the cycles
 are restricted to have length at least $2$, one obtains the \textit{derangement numbers} $d_{n}$, given by 
\begin{equation}
\label{der-form1}
d_{n} = n! \sum_{j=0}^{n} \frac{(-1)^{j}}{j!}.
\end{equation}
\noindent
They satisfy the  recurrence $d_{n} = nd_{n-1} + (-1)^{n}$, with $d_{0}=1$. It follows from 
\eqref{der-form1} that 
\begin{equation}
\label{limit-der}
\lim\limits_{n \rightarrow \infty} \frac{d_{n}}{n!} = \frac{1}{e}.
\end{equation}
\noindent
Broder \cite{broder-1984a}
introduced the notion of $r$-\textit{permutations}. For $r \in \mathbb{N}$ an $r$-permutation of
$n+r$ is a permutation where the first $r$ elements, called \textit{special},  are in distinct
 cycles. The number of
$r$-permutations of $[n+r]$ into $k+r$ cycles are counted by the $r$-\textit{Stirling numbers of the
first kind}, denoted by $\begin{bmatrix} n \\ k \end{bmatrix}_{r}$.  An $r$-\textit{derangement} on
$[n+r]$ is an $r$-permutation without fixed points.  Information about these concepts appears in
\cite{mezo-2019b,wang-2017a}. These concepts are now extended to signed permutations.

\begin{definition}
\label{def1}
A \textit{derangement of type} $B$ on $[n]$ is a signed permutation $\sigma$ such that
$\sigma(i) \neq i$ for every $i \in [n]$. The set of all such permutations is denoted by
$\mathcal{D}_{n}^{B}$ and its cardinality  by $d_{n}^{B}$.
\end{definition}

Chow \cite{chow-2006a} (see also Assaf \cite{assaf-2010a}) proved that
\begin{equation}
\label{chow-1}
d_{n}^{B} = n! \sum_{k=0}^{n} \frac{(-1)^{k} 2^{n-k}}{k!},
\end{equation}
\noindent
and the analog of \eqref{limit-der} is
\begin{equation}
\label{limit-derB}
\lim\limits_{n \rightarrow \infty} \frac{d_{n}^{B}}{n! \, 2^{n} } = \frac{1}{\sqrt{e}}.
\end{equation}
\noindent
This sequence appears as $A000354$ in OEIS and its first few values (starting at $n=0$) are
\begin{equation}
1, \, 1, \, 5, \, 29, \, 233, \, 2329, \, 27949, \, 391285, \, \cdots.
\end{equation}
\noindent
Formula  \eqref{chow-1} is equivalent to 
\begin{equation}
\sum_{n=0}^{\infty} d_{n}^{B} \frac{x^{n}}{n!} = \frac{e^{-x}}{1-2x}.
\end{equation}

The main object of the present work is introduced next.

\begin{definition}
Let $n, \, r \in \mathbb{N}$. A \textit{type} B $\,\, r$-\textit{derangement} on the set $[n+r]$ is a
signed permutation on $[n+r]$, without fixed points and with $r$ elements (called \textit{special})
restricted to be in distinct cycles.  The set of all $r$-derangements of type $B$ on $[n+r]$ is denoted
by $\mathcal{D}_{n,r}^{B}$. Its  cardinality is denoted by $d_{n,r}^{B}$.  The case $r=0$  recovers
$d_{n}^{B}$ in Definition \ref{def1}.
\end{definition}

The number of elements of
$\mathcal{D}_{n,r}^{B}$ with $k+r$ cycles is called the $r$-\textit{Stirling number of type} $B$  and
is denoted by
$\begin{displaystyle} \begin{bmatrix} n \\ k \end{bmatrix}_{\geq 2, r}^{B} \end{displaystyle}$.
Counting over all possible cycles gives the relation
\begin{equation}
d_{n,r}^{B} = \sum_{k=0}^{n} \begin{bmatrix} n \\ k \end{bmatrix}_{\geq 2, r}^{B}.
\end{equation}

\begin{example}
The permutation 
$\sigma = (1 \overline{7} 4)(\overline{2})( 3 \overline{6} 5)$ is a type B$\, r$-derangement for
$0 \leq r \leq 3$ on the
set $[7]$.
\end{example}

\begin{Note}
\label{note-initcond}
The case $r=0$ has been discussed in \cite{mezo-2019a}.  The recurrence
\begin{equation}
\begin{bmatrix} n +1 \\ k \end{bmatrix}_{\geq 2, 0}^{B} = 2n
\begin{bmatrix} n  \\ k \end{bmatrix}_{\geq 2, 0}^{B}  +
2n \begin{bmatrix} n -1 \\ k -1 \end{bmatrix}_{\geq 2, 0}^{B}  +
\begin{bmatrix} n  \\ k - 1  \end{bmatrix}_{\geq 2, 0}^{B},
\end{equation}
\noindent
with the initial conditions
\begin{equation}
\begin{bmatrix} n \\ 0 \end{bmatrix}_{\geq 2, 0}^{B} = \delta_{n,0} \,\, \text{for} \,\, n>0
\,\, \text{ and } \begin{bmatrix} n \\ k \end{bmatrix}_{\geq 2, 0}^{B} = 0 \text{ for } n, \, k < 0,
\end{equation}
\noindent
is established there.
\end{Note}

\section{A recurrence for the $r$-Stirling numbers of type $B$}
\label{sec-recurrence}

This section  presents a recurrence for the $r$-Stirling numbers of type $B$.  The
initial condition involve the \textit{Lah numbers} $L(n,k)$,  defined as the number of ways a
set of $n$ elements can be partitioned into $k$ nonempty linearly ordered subsets \cite{riordan-2002a}.

\begin{theorem}
\label{rec-short}
For $n \geq 0$ and $k, \, r \geq 1$ the recursion
\begin{eqnarray*}
\begin{bmatrix} n+1 \\ k \end{bmatrix}_{\geq 2, r}^{B}  & = &
\begin{bmatrix} n \\ k-1 \end{bmatrix}_{\geq 2, r}^{B}
+
2n! \sum_{j=1}^{n} \frac{2^{j}}{(n-j)!}
\begin{bmatrix} n-j \\ k-1 \end{bmatrix}_{\geq 2, r}^{B}   \\
& & \quad \quad \quad \quad \quad \quad \quad +
4r n! \sum_{j=0}^{n} \frac{(j+1) 2^{j}}{(n-j)!}
\begin{bmatrix} n -j \\ k \end{bmatrix}_{\geq 2, r-1}^{B}    \\
& = &
\begin{bmatrix} n \\ k-1 \end{bmatrix}_{\geq 2, r}^{B}
+ 4r
\begin{bmatrix} n \\ k  \end{bmatrix}_{\geq 2, r-1}^{B}  \\
& & \quad \quad \quad \quad + 4n! \sum_{j=1}^{n} \frac{2^{j-1}}{(n-j)!}
\left(
\begin{bmatrix} n-j \\ k-1   \end{bmatrix}_{\geq 2, r}^{B}
+ 2r(j+1)
\begin{bmatrix} n-j \\ k   \end{bmatrix}_{\geq 2, r-1}^{B}
\right),
\nonumber
\end{eqnarray*}
\noindent
holds. The initial conditions are given by
\begin{equation*}
\begin{bmatrix} n \\ 0 \end{bmatrix}_{\geq 2, r}^{B}   = 2^{n} n! \sum_{j=0}^{r}
\binom{r}{j} \binom{n-1}{r-j-1} 2^{r-j} \quad \text{ and } \quad
\begin{bmatrix} 0 \\ 0 \end{bmatrix}_{\geq 2, r}^{B} = 1.
\end{equation*}
\end{theorem}
\begin{proof}
The initial condition is established first. Recall  that 
$\begin{displaystyle} \begin{bmatrix} n \\ 0   \end{bmatrix}_{\geq 2, r}^{B} \end{displaystyle}$
counts the number of signed permutations of $[n+r]$, with $r$ cycles, without fixed points and
$r$ special elements restricted to be in distinct cycles.  Therefore each cycle contains a single
special point and,  since there are no fixed points, the special points appearing in a cycle of
length $1$ must be colored.
 The next step is to place the $n$ non-special points in the $r$ cycles. Let $j$ be the
number of cycles that receive no new points ($0 \leq j \leq r$).  Then
\begin{equation}
 \begin{bmatrix} n \\ 0   \end{bmatrix}_{\geq 2, r}^{B}  =
 \sum_{j=0}^{r} \binom{r}{j} \cdot 2^{n+r-j} (r-j)! \, L(n,r-j).
 \end{equation}
 \noindent
 The binomial coefficient $\binom{r}{j}$ takes into account the choices of cycles without new points, the
 factor $(r-j)!$ takes care of the ordering of the cycles with new points, the Lah number counts the
 ways to partition the $n$ non-special points in order to place them in the $r-j$ cycles and finally the
 power of $2$ describes whether to color the remaining elements or not. The form of the initial
 condition stated above is obtained by using the expression
 \begin{equation}
 L(n,k) = \binom{n}{k} \binom{n-1}{k-1} (n-k)! = \frac{n!}{k!} \binom{n-1}{k-1},
 \end{equation}
 \noindent
 appearing in \cite{riordan-2002a}.

 \smallskip

 To prove the general case of the recurrence, recall that there are  $r+k$
 cycles of size at least $2,  \,\,  r$ containing the special points and $k$ containing no special points.
 Now consider cases  for the image of  non-special element $n+1$ under $\pi$.  \\

 \noindent
 \texttt{Option 1}. The element $\overline{n+1}$ is fixed by  $\pi$. Recall the
 warning: this means that $\pi(n+1) = \overline{n+1}$ and $\pi(\overline{n+1}) = n+1$.  Removing the
 cycle  $\{ n+1, \, \overline{n+1}\}$ leaves 
 $ \begin{bmatrix} n \\ k-1   \end{bmatrix}_{\geq 2, r}^{B}$ choices. This gives the first term in the
 recurrence. \\

 \noindent
 \texttt{Option 2}. Assume $n+1$ is the image under a signed permutation of an uncolored symbols.
 There are two cases to consider:
 
 \noindent
  \textit{Case 1}: $n+1$ belongs to one the first $r$ cycles. Then the 
  cycle  $\mathfrak{b}$ containing $n+1$
 has, in addition to $n+1$,  one special element and $j$ additional  ones $(0 \leq j \leq n$)
 from $\{ 1, \, 2, \, \ldots, n \}$, for  a total of $j+2$ elements.  Since $n+1$ is not colored, there are
 $2^{j+1}$ ways to color or not the other elements in $\mathfrak{b}$. The $j+2$ elements of the
 cycle $\mathfrak{b}$ can now be permuted cyclically in $(j+1)!$ ways. The count is now
 \begin{equation}
 r \sum_{j=0}^{n} \binom{n}{j} (j+1)! 2^{j+1} \begin{bmatrix} n-j \\ k   \end{bmatrix}_{\geq 2, r-1}^{B}.
 \end{equation}
 \noindent
 In this count, the number of special cycles has been reduced by $1$, since $n+1$
 occupies one of them. 
 
 \noindent
 \textit{Case 2}. If $n+1$ does not belong to one of the special cycles, then the 
 number in Case 1  becomes
 \begin{equation}
  \sum_{j=1}^{n} \binom{n}{j} j! 2^{j} \begin{bmatrix} n-j \\ k-1   \end{bmatrix}_{\geq 2, r}^{B}.
 \end{equation}
 \noindent
 The index $j$ counts the number of additional elements in the cycle containing $n+1$ and the
 remaining terms are interpreted as before.  \\

 \noindent
 \texttt{Option 3}. Assume $\overline{n+1}$ belongs to one of the  cycles. Then the 
  count is similar as in Option 2. 
 The numbers of cases where  $\overline{n+1}$ is in one of $r$ special cycles is
 \begin{equation}
 2r \sum_{j=0}^{n} \binom{n}{j} (j+1)! 2^{j} \begin{bmatrix} n-j \\ k   \end{bmatrix}_{\geq 2, r-1}^{B}.
 \end{equation}
 \noindent
 Otherwise, $\overline{n+1}$ is in a cycle without  special points and this contributes 
 \begin{equation}
  \sum_{j=0}^{n} \binom{n}{j} j! 2^{j} \begin{bmatrix} n-j \\ k -1  \end{bmatrix}_{\geq 2, r}^{B}
 \end{equation}
 \noindent
to the count. This completes the proof.
\end{proof}

\begin{Note}
The recurrence is now  used to generate the values of
$\begin{bmatrix} n \\ k   \end{bmatrix}_{\geq 2, r}^{B}.$ The matrix below gives the values
for $r=3$ in the range $0 \leq n, \, k \leq 6$: \\

\begin{equation}
\label{matrix-1}
\left[ \begin{bmatrix} n \\ k   \end{bmatrix}_{\geq 2, 3}^{B} \right] =
\left(
\begin{array}{ccccccc}
 1 & 0 & 0 & 0 & 0 & 0 & 0 \\
 12 & 1 & 0 & 0 & 0 & 0 & 0 \\
 144 & 28 & 1 & 0 & 0 & 0 & 0 \\
 1824 & 592 & 48 & 1 & 0 & 0 & 0 \\
 25344 & 11232 & 1552 & 72 & 1 & 0 & 0 \\
 391680 & 213888 & 41824 & 3280 & 100 & 1 & 0 \\
 6727680 & 4267008 & 1061248 & 119520 & 6080 & 132 & 1 \\
\end{array}
\right).
\end{equation}
\end{Note}

\begin{Note}
The special value $d_{r,0} = 1$ comes directly from the definition. The recurrence in Theorem
\ref{rec-short}  is now used to check that $d_{1,r}^{B} = 1+4r$. Indeed,
\begin{equation}
d_{r,1}^{B} = \begin{bmatrix} 1 \\ 0   \end{bmatrix}_{\geq 2, r}^{B}  +
\begin{bmatrix} 1 \\ 1   \end{bmatrix}_{\geq 2, r}^{B},
\end{equation}
\noindent
and the initial condition in Theorem \ref{rec-short} shows that 
\begin{equation}
\begin{bmatrix} 1 \\ 0   \end{bmatrix}_{\geq 2, r}^{B} = 2^{1} \cdot 1! \times
\sum_{j=0}^{r-1} \binom{r}{j} \binom{0}{r-j-1} 2^{r-j} = 4r,
\end{equation}
\noindent
and the second term  in the recurrence yields
\begin{equation}
\begin{bmatrix} 1 \\ 1   \end{bmatrix}_{\geq 2, r}^{B} =
\begin{bmatrix} 0 \\ 0   \end{bmatrix}_{\geq 2, r}^{B}  + 4r
\begin{bmatrix} 0 \\ 1   \end{bmatrix}_{\geq 2, r}^{B} + \text{an empty sum} = 1,
\end{equation}
\noindent
for a total of  $d_{1,r}^{B} = 1+4r$.
\end{Note}

\section{An approach from  the theory of Riordan arrays }
\label{sec-arrays}

This section considers  the matrix $\mathcal{C}_{\ge 2, r}:=\left({n \brack k}_{\ge 2, r}^B\right)_{n, k \geq 0}$
  as a Riordan array.  Enumerative arguments for some  combinatorial identities are presented. The
   unsigned entries of the matrix $\C^{-1}_{\ge 2, r}$ are discussed first.

Recall that an infinite lower triangular matrix  $L=\left[l_{n,k}\right]_{n,k\geq 0}$ is called an
\emph{exponential  Riordan array}  \cite{barry-2007a} if its $k^{th}$-column has generating function of 
the form 
 $g(z)\left(f(z)\right)^k/k!, k = 0, 1, 2, \ldots g(z)$, where  $g(z)$ and $f(z)$ are formal power series with
  $g(0) \neq 0, \, f(0)=0$ and $f'(0)\neq 0$.  The matrix corresponding to the pair $f(z), g(z)$ is denoted 
  by  $(g(z),f(z))$. 
  
   Multiplying  $(g, f)$ by a column vector $(c_0, c_1, \dots)^t$ with exponential generating function $h(z)$,  
   results in a  column vector with exponential generating function $g\cdot h\circ f$. This property is known 
   as the fundamental theorem of exponential Riordan arrays or \textit{summation property} 
   \cite{deutsch-2009a}. The product 
   of two exponential Riordan arrays $(g(z),f(z))$ and $(h(z),\ell(z))$ is defined by:
 $$(g(z),f(z))*(h(z),\ell(z))=\left(g(z)h\left(f(z)\right), \ell\left(f(z)\right)\right).$$
The set of all exponential Riordan matrices is a group  under this operation
\cite{barry-2007a}, \cite{shapiro-1991a}. The identity is $I=(1,z)$ and
 \begin{equation}\label{invRiordan}
 (g(z), f(z))^{-1}=\left(\frac{1}{\left(g\circ \overline{f}\right)(z)}, \overline{f}(z)\right),
\end{equation}
here $\overline{f}(z)$ denotes  the compositional inverse of $f(z)$; that is, $(\overline{f} \circ f)(z) = z$.

Deutsch et al. \cite{deutsch-2009a} gave an algorithm to calculate the entry $l_{n,k}$ of an exponential Riordan array. They proved
that every element $l_{n+1,k}$ of an exponential Riordan array  can be expressed as a linear 
combination of the elements in the preceding row. In particular, for 
 $L=\left[l_{n,k}\right]_{n,k\geq 0}=(g(z),f(z))$ there are sequences $(a_n)$ and  $(z_n)$ such that
\begin{align}
l_{n+1,0} &= \sum_{i\geq 0} i! a_i l_{n,i}, \label{recriordan1}\\
l_{n+1,k}&=a_0l_{n,k-1} + \frac{1}{k!}\sum_{i\geq k}i!(z_{i-k}+ka_{i-k+1})l_{n,i}, \label{recriordan2}
\end{align}
where
\begin{equation*}
    A(t)=a_0+a_1t+a_2t^2+\cdots ,\quad Z(t)=z_0+z_1t+z_2t^2+\cdots
\end{equation*}
satisfy the functional relations
\begin{align} \label{production1}
A(t)=f'(\bar{f}(t)), \quad Z(t)=\frac{g'(\bar{f}(t))}{g(\bar{f}(t))}.
\end{align}
Conversely, \eqref{production1} implies \eqref{recriordan1} and \eqref{recriordan2}. 

The next results states that the matrix defined by  
$\mathcal{C}_{\ge 2, r}:=\left({n \brack k}_{\ge 2, r}^B\right)_{n, k \geq 0}$ is an exponential Riordan array.

\begin{theorem}
\label{teoriob}
The matrix $\mathcal{C}_{\ge 2, r}$ is the exponential Riordan array given by
$$\left(\left(\frac{1+2z}{1-2z}\right)^r, -\ln(1-2z)-z\right).$$
\end{theorem}
\begin{proof}
Let $C_k(x):=\sum_{n\geq 0} {n+k \brack k}_{\ge2, r}^B \frac{x^n}{n!}$ be the generating function of the 
columns of the matrix $\mathcal{C}_{\ge 2, r}$. The initial values in Theorem  \ref{rec-short} shows 
that the generating function of the first column is given by
\begin{align*}
C_0(x)&=\sum_{n\geq 0} 2^nn! \sum_{j=0}^r\binom rj \binom{n-1}{r-j-1}2^{r-j}\frac{x^n}{n!}=\sum_{j=0}^r \binom rj 2^{r-j} \sum_{n\geq 0} \binom{n-1}{r-j-1}(2x)^n.
\end{align*}
The identity
$$\frac{1}{(1-x)^t}=\sum_{n\geq 0} \binom{t+n-1}{t-1}x^n,$$
gives
$$\left(\frac{2x}{1-2x}\right)^{r-j}=\sum_{n\geq 0} \binom{n-1}{r-j-1}x^n,$$
Therefore
\begin{align*}
C_0(x)=\sum_{j=0}^r \binom rj 2^{r-j} \left(\frac{2x}{1-2x}\right)^{r-j}=\left(1 + \frac{4x}{1-2x}\right)^r=\left(\frac{1+2x}{1-2x}\right)^r.
\end{align*}
The generating function for the column $C_1(x)$ is presented next.  The relation 
\begin{align}\label{eccol1}
 {n \brack 1}_{\ge2, r}^B=\sum_{k=2}^n\binom{n}{k}(k-1)!2^k {n-k \brack 0}_{\ge2, r}^B + n{n-1 \brack 0}_{\ge2, r}^B.
 \end{align}
follows from the following combinatorial argument. Assume a  non-special cycle has $k$ elements
 ($2\leq k \leq n$). These can be chosen and colored  in $\binom nk (k-1)!$ ways. The coloring of 
 these these elements produce $2^{k}$ options. The remaining elements can be organized in   ${n-k \brack 0}_{\ge2, r}^B$ ways. Summing over $k$ gives  the first expression in \eqref{eccol1}. On the other hand, in 
 the case where a non-special block has only one element gives 
  $ n{n-1 \brack 0}_{\ge2, r}^B$ options.  Therefore 
$$C_1(x)=C_0(x) \sum_{n\geq 2} \frac{(2x)^n}{n} + xC_0(x)=C_0(x)(-x-\log(1-2x)).$$
Similarly
\begin{align}\label{columns}
C_k(x)=C_0(x)\frac{(-x-\log(1-2x))^k}{k!},
\end{align}
since $C_0(x)$ takes into account the $r$ special cycles and  $(-x-\log(1-2x))^k/k!$ are sequences of 
unordered  $k$ cycles. The result now follows 
from \eqref{columns} and the definition of a Riordan array.
\end{proof}

\section{Counting distances on graphs}
\label{sec-graphs}

Given a graph $G$ the sequence 
$S_{G}(n)$ is defined as the  number of vertices which have a distance $n$ from a given vertex in $G$. 
This sequence was discussed in \cite{baake-1997a} and  \cite{conway-1997b}. The example 
$S_{\mathbb{Z}^2}(3)=12$ is pictured in  Figure \ref{lattice}.
\begin{figure}[htbp]
\centering
\includegraphics[scale=0.8]{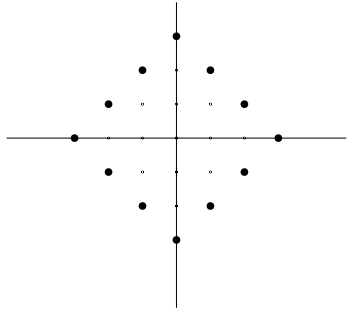}
\caption{Points at a distance 2 from a fixed point.} \label{lattice}
\end{figure}
\noindent
In the case of  the $r$-dimension  lattice $\mathbb{Z}^r$, it follows that
$$\sum_{n\geq 0} S_{\mathbb{Z}^r}(n) x^n=\left(\frac{1+x}{1-x}\right)^r.$$

The first column in the Riordan array given in  Theorem \ref{teoriob}  gives the next 
statement.  A combinatorial proof is presented. 

\begin{theorem}
For $r,n\in \mathbb{N}$
\begin{equation}
{n\brack 0}_{\geq 2,r}^B=2^nn!S_{\mathbb{Z}^r}(n).
\end{equation}
\end{theorem}
\begin{proof}
A permutation $\sigma$ with no non-special elements is  written as
$\sigma =c_1c_2\cdots c_r,$ where $c_i$ are the special cycles containing one of the $[r]$ 
elements. Define $x_{\sigma}=(x_1,\dots ,x_r)\in \mathbb{Z}^r$  by
$$x_i=\begin{cases}
      0, & \text{ if } \sigma _i = \bar i; \\
      -(ord(c_i)-1), &\text{ if }   \bar i \in c_i;\\
      ord(c_i)-1, & \text{otherwise.}
   \end{cases}$$
   where $ord(c_i)$ is the length of the cycle $c_{i}$. It is clear that $\sum _{i=1}^r|x_i|=n$, so 
   $x_{\sigma}$ is at distance $n$ from the origin. This construction does not take into account the sign
    in the $[n]$ elements. This can be done in  $2^n$ ways. Finally, the $n!$ term count the ordering of 
    position in the cycles. The proposition follows.
\end{proof}

 Lengyel \cite{lengyel-2006a} presented the interesting expression:
$$\frac{1+x}{1-x}=\sum_{n=1}^{\infty} x^{\lfloor \frac{n}{\phi}\rfloor} + \sum_{n=1}^{\infty} x^{\lfloor n\phi^2\rfloor}, \quad \text{for} \, |x|<1.$$
Here $\phi$ is the golden ratio.

The next statement characterizes the two diagonals below the main one.

\begin{theorem}
\label{diagonals} 
For $n\geq 0$,  the identities
\begin{align*}
    { n + 1 \brack n }^B_{\geq 2,r}&=2(n+1)(n+2r),\\
    {n+2 \brack n}^B_{\geq 2, r}&=
    \frac{4}{3}\binom{n+2}{2}(3n^2+n+12nr+12r^2),
\end{align*}
\noindent
hold.
\end{theorem}
\begin{proof}
There are two cases to consider, either
\begin{itemize}
    \item There is an element that belongs to one of the $r$ special cycles.  This happens in 
    $(n+1)r2^{2}$ ways; \texttt{or}, 
    \item Two elements form a non special cycle. This can be done in  $2^2\binom{n+1}{2}$ ways.
\end{itemize}
The first identity follows from here. The proof of the second identity  is similar.
\end{proof}

\subsection{Combinatorial interpretation of the inverse matrix}
The matrix $\C_{\geq 2, r}$ is a Riordan array and since the set of all Riordan matrices is a group, its 
inverse exists. This subsection presents a  combinatorial interpretation of the unsigned inverse matrix 
of $\C_{\geq 2, r}$ in terms of plane increasing trees. The particular case $r=0$ was presented in 
\cite{mezo-2019a}. The notation  $\T_{\geq 2, r}:=\left[T_{\geq 2, r}(n,k)\right]_{n, k \geq 0}$ is used 
for the inverse of $\C_{\geq 2, r}$.  Then  \eqref{invRiordan} shows  that  $\T_{\geq 2, r}$ is the
exponential Riordan array given by
$$\T_{\geq 2, r}=\left(\left(-1 + \frac{1}{1+W\left(-\frac{1}{2}e^{-\frac 12 - z}\right)}\right)^r, \frac{1}{2} +  W\left(-\frac{1}{2}e^{-\frac{1}{2}-z}\right)\right),$$
with  $W(z)$ is the Lambert $W$ function. This is defined by
\begin{equation}
W(x)e^{W(x)}=x,
\end{equation}
\noindent
for all real (or complex) $x$. More information on this function appears in \cite{corless-1996a}.

Let $\Tc_{\geq 2, r}:=\left[\overline{T}_{\geq 2, r}(n,k)\right]_{n, k \geq 0}$  be the unsigned matrix of $\T_{\geq 2, r}$, that is
\begin{align*}
\Tc_{\geq 2, r}:=&\left[\overline{T}_{\geq 2, r}(n,k)\right]_{n, k \geq 0}=\left[(-1)^{n+k} T_{\geq 2, r}(n,k)\right]_{n, k \geq 0}\\
=&\left(\left(-1 + \frac{1}{1+W\left(-\frac{1}{2}e^{-\frac 12 + z}\right)}\right)^r, -\frac{1}{2} -  W\left(-\frac{1}{2}e^{-\frac{1}{2}+z}\right)\right).
\end{align*}
The first few rows of the matrix $\Tc_{\geq 2, 3}=\left[\overline{T}_{\geq 2, 3}(n,k)\right]_{n, k \geq 0}$ are
$$
\left(
\begin{array}{cccccccc}
 1 & 0 & 0 & 0 & 0 & 0 & 0 \\
 12 & 1 & 0 & 0 & 0 & 0 & 0 \\
 192 & 28 & 1 & 0 & 0 & 0 & 0 \\
 3936 & 752 & 48 & 1 & 0 & 0 & 0 \\
 99456 & 22304 & 1904 & 72 & 1 & 0 & 0 \\
 3001344 & 748672 & 76320 & 3920 & 100 & 1 & 0 \\
 105544704 & 28412416 & 3265792 & 203040 & 7120 & 132 & 1 \\
\end{array}
\right).
$$
Now let $\overline{T}_k(z)$ be the generating function of the $k^{th}$ column of $\Tc_{\geq 2, r}$. Then
$$\overline{T}_0(z)=\left(-1 + \frac{1}{1+W\left(-\frac{1}{2}e^{-\frac 12 + z}\right)}\right)^r= \left(F'(z)\right)^r,$$
where
\begin{equation}
F(z):=- \tfrac{1}{2} -  W\left(-\tfrac{1}{2}e^{-\frac{1}{2}+z}\right).
\end{equation}
\noindent
The generating function $F(z)$ counts  the total number of labeled rooted trees of subsets of an 
$n$-set \cite{mcmorris-1981a}. In the sequence A005172 of OEIS, Peter Bala states  that $F(z)$  satisfies the 
functional equation
\begin{equation}
x=\int_{0}^{F(x)}\frac{1}{\Phi(t)}dt,
\end{equation}
where $\Phi(t)=(1+2t)/(1-2t)$. Then  \cite[Theorem 1]{bergeron-1992a} provides an alternative combinatorial 
interpretation for this sequence. It turns out that this sequence also counts the plane increasing trees on
 $n$ vertices, such that each vertex of out-degree $k\geq 1$ is colored with one of $2^{k+1}$ colors.   
Therefore $F'(z)$, the derivative of $F(z)$, counts plane increasing trees on $n+1$ vertices.

The power series expansion of $F'(z)$ starts as 
$$F'(z)=1+ 4 \frac{z}{1!} + 32 \frac{z^2}{2!} + 416 \frac{z^3}{3!} +   7552 \frac{z^4}{4!} + 176128 \frac{z^5}{5!}  + \cdots  $$
Figure \ref{trees} shows the corresponding trees with $3$ vertices. The parenthesis denotes the possible
 colorings for a vertex.
\begin{figure}[htbp]
\centering
\includegraphics[scale=1]{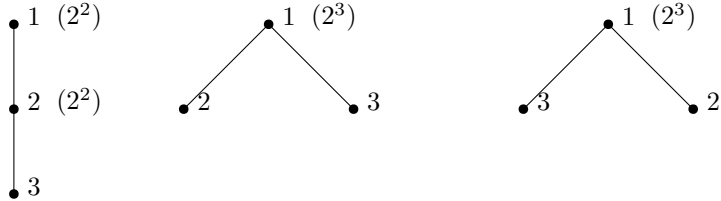}
\caption{Colored plane increasing trees on $3$ vertices.} \label{trees}
\end{figure}

The identity 
\begin{equation}
\overline{T}_k(z)=\left(F'(z)\right)^r\frac{1}{k!}\left(F(z)\right)^k,
\end{equation}
\noindent
leads to  the next  result.

\begin{theorem}
The sequence $\overline{T}_{\geq 2, r}(n,k)$ counts the number of  $r$-ordered and $k$-unordered 
elements in a $r+k$-forest  of increasing trees on $n$ vertices, such that each vertex of out-degree 
$k\geq 1$ is colored with one of $2^{k+1}$ colors and the $r$ ordered connected components are rooted.
\end{theorem}

The next recurrence now follows from  \eqref{recriordan1} and \eqref{recriordan2}.

\begin{theorem}\label{rec1inversa}
The sequence $\overline{T}_{\geq 2, r}(n,k)$ for $n\geq k\geq 0$ is determined by the recurrence
$$\overline{T}_{\geq 2, r}(n+1,k) = \overline{T}_{\geq 2, r}(n,k-1) + \frac{1}{k!}\sum_{i=k}^ni!2^{i-k+2}((i-k+1)r + k)\overline{T}_{\geq 2, r}(n,i),$$
and the initial conditions $\overline{T}_{\geq 2, r}(n,0)=\delta_{n,0}$ for $n\geq 0$, and $\overline{T}_{\geq 2, r}(n,k)=0$ for $n, k<0$.
\end{theorem}

\section{Properties of the  numbers $d_{r,n}^{B}$}
\label{sec-drnb}

This section discusses the expression
\begin{equation}
d_{r,n}^{B} = \sum_{k=0}^{n} \begin{bmatrix} n \\ k   \end{bmatrix}_{\geq 2, r}^{B},
\end{equation}
\noindent
for the total number of $r$-derangements of type $B$ on $[n+r]$.  This is the row sum of the
matrix in \eqref{matrix-1}. The case $r=0$ appears in  \eqref{chow-1}.

\begin{theorem}
The recurrence
\begin{equation}
\label{diff1}
d_{r,n}^{B} = d_{r-1,n}^{B} + 2nd_{r,n-1}^{B} + 2nd_{r-1,n-1}^{B},
\end{equation}
\noindent
holds. It is supplemented by the initial conditions $d_{r,0}^{B} =1$ and $\begin{displaystyle}
d_{0,n}^{B} = n! \sum_{k=0}^{n} \frac{(-1)^{k} }{k!} 2^{n-k}. \end{displaystyle}$
\end{theorem}
\begin{proof}
Consider a permutation $\pi$ in
$\mathcal{D}_{n,r}^{B}$.  There are  $r$ special elements and assume that
$1$ is the first of them. The recurrence is obtained by distinguishing cases according to the type
of cycle $\mathfrak{c}$ containing $1$. The initial conditions appeared in
Note \ref{note-initcond}. \\

\noindent
\texttt{Case 1}. Assume $\mathfrak{c}$ is of length $1$. Then the cycle must be $( \overline{1} )$, since
there are no fixed points in $\pi$.  The rest of the permutation is in $\mathcal{D}^{B}_{n,r-1}$, counting 
for the term $d_{n,r-1}^{B}$.  \\

\noindent
\texttt{Case 2}. If  $\mathfrak{c}$ is of length $2$, then $\mathfrak{c}$ is of the form
$(1 \, x )$ with $x$ non-special. Then $1$ and $x$ can be  colored or not, for a total of
$n$ choices. The rest of the permutation is in $\mathcal{D}^{B}_{n-1,r-1}$, for  a total of
$4nd_{n-1,r-1}^{B}$ choices. \\

\noindent
\texttt{Case 3}. Finally consider the case when the cycle $\mathfrak{c}$ has at least $3$ elements.  To
produce such a  permutation, choose $k$ to be non-special and drop it from the list of non-special
elements. Then form an arbitrary $B$-$r$-derangement on $n-1+r$ elements in $d_{n-1,r}^{B}$ ways.
After that, insert 
$k$ to the right of $1$ and optionally color it. This accounts for  $2n d_{n-1,r}^{B}$ choices. In this process,
there are certain permutations that have been counted twice. Namely, those for which $1$ is in 
a cycle of length $1$ for a permutation in $\mathcal{D}_{n-1,r}^{B}$. Inserting $k$ then
produces a cycle of length $2$, already counted in Case 2. Therefore these must be 
excluded. Thus, the total count is $2n \left( d_{n-1,r}^{B} - d_{n-1,r-1}^{B} \right)$. \\

The proof is complete.
\end{proof}

\begin{Cor}
For fixed $n \in \mathbb{N}$, the function $d_{n,r}^{B}$ is a polynomial in $r$ of degree $n$.
\end{Cor}
\begin{proof}
This follows from $d_{r,0} = 1$ and \eqref{diff1}, written in the form
\begin{equation}
\label{diff2}
d_{r,n}^{B} - d_{r-1,n}^{B} =  2n( d_{r,n-1}^{B} + d_{r-1,n-1}^{B}).
\end{equation}
\end{proof}

\begin{example}
The value $d_{r,1}^{B} = 1 + 4r$ and the difference equation \eqref{diff2} give
\begin{equation}
d_{r,2}^{B}  - d_{r-1,2}^{B} = 4( 1+4r + 1+4(r-1)) = 32r-8.
\label{diff3}
\end{equation}
\noindent
Therefore $d_{r,2}^{B}$ is a quadratic polynomial in $r$.  The ansatz
$d_{r,2}^{B} = a_{0}r^{2}+a_{1}r+a_{2}$ in \eqref{diff3} gives $a_{2} = 16$ and $a_{1} = 8$. The
remaining coefficient comes  from \eqref{chow-1}; so that 
\begin{equation}
d_{r,2}^{B} = 16r^{2} + 8r+5.
\end{equation}
\noindent
Similarly,
\begin{eqnarray}
& & \\
d_{r,3}^{B} & = & 64r^{3}+48r^2+92r+29,  \nonumber \\
d_{r,4}^{B} & = & 256r^{4}+256r^3+992r^2+592r+233, \nonumber  \\
d_{r,5}^{B} & = & 1024r^5+1280r^4+8320r^3+7200r^2+7796r+2329, \nonumber \\
d_{r,6}^{B} & = & 4096r^6+6144r^5+60160r^4+67840r^3+141424r^2+83672r+27949.  \nonumber
\end{eqnarray}
\end{example}

An explicit formula for  $d_{r,n}^{B}$ is presented next.

\begin{theorem}\label{explicitform}
For all $n, \, r \geq 0$,
\begin{equation}
d_{r,n}^{B} = 2^{n} \sum_{i=0}^{r} \binom{r}{i} n^{\underline{i}} 2^{i}
\sum_{k=0}^{n-i} \binom{n-i}{k}(-1)^{k} (i+1)^{\overline{n-i-k}} 2^{-k}.
\end{equation}
\noindent
Here $n^{\underline{i}} = n(n-1) \cdots (n-i+1)$ and $n^{\overline{j}} = n(n+1) \cdots (n+j-1)$.
\end{theorem}
\begin{proof}
We consider cases with  $r-i$ special elements ($0\le i\le r$) are in cycles of length one (and therefore 
necessarily barred). The other $i$ special elements are in cycles of length greater than one. Choose 
first these special elements (in $\binom{r}{i}$ ways), and then place the $i$ non-special elements into their 
cycles. This gives  $n(n-1)\cdots(n-i+1)=n^{\underline{i}}$ choices. Then mark both the $i$ special elements 
and the non-special ones with a bar, for a total of $2^i\cdot2^i$ options. The other $n-i$ non-special 
elements are placed into the $i$ special cycles or outside  them. This yields $(i+1)^{\overline{n-i}}$ options. 
There are $2^{n-i}$ ways to  mark these elements or not. Up to now, the count is
\begin{equation}
\binom{r}{i}n^{\underline{i}} 2^{2i}(i+1)^{\overline{n-i}}2^{n-i}\label{kzero}.
\end{equation}
Observe that in the non-special part of the permutation, the placement above  might have introduced 
some fixed points and these must be removed. Among the above 
permutations, the number of those type B derangements which have exactly $k$ fixed points among
 the $n-i$ elements is
\[\binom{n-i}{k}n^{\underline{i}}2^{2i}(i+1)^{\overline{n-i-k}}2^{n-i-k}.\]
The inclusion-exclusion principle now completes the proof.
\end{proof}

\subsection{The generating function and asymptotics of $d_{r,n}^{B}$}
Theorem \ref{explicitform} is now  used to produce  the exponential generating 
function of  $\{ d_{r,n}^{B} \}$. 

\begin{theorem}
The formula
\begin{equation}
\sum_{n=0}^\infty d_{r,n}^{B} \frac{x^n}{n!}=\frac{e^{-x}}{1-2x}\left(\frac{1+2x}{1-2x}\right)^r,
\end{equation}
\noindent
holds.
\end{theorem}
\begin{proof}
Start with 
\[\sum_{n=0}^\infty\frac{d_{r,n}^{B}}{2^n}\frac{x^n}{n!}=\sum_{i=0}^r\binom{r}{i}2^{i}\sum_{k=0}^\infty\left(-\frac12\right)^k\sum_{n=i+k}^\infty\binom{n-i}{k}n^{\underline{i}}(i+1)^{\overline{n-i-k}}\frac{x^n}{n!}.\]
The innermost sum is
\[\sum_{n=i+k}^\infty\binom{n-i}{k}n^{\underline{i}}(i+1)^{\overline{n-i-k}}\frac{x^n}{n!}=\frac{x^i}{(1-x)^{i+1}}\frac{x^k}{k!}.\]
Multiplying the right hand side by $\left(-\frac12\right)^k$ and summing over $k$ produces 
 $e^{-x/2}$. Thus
\[\sum_{n=0}^\infty\frac{d_{r,n}^{B}}{2^n}\frac{x^n}{n!}=\sum_{i=0}^r\binom{r}{i}2^{i}\frac{e^{-x/2}x^i}{(1-x)^{i+1}}=\frac{e^{-x/2}}{1-x}\left(\frac{1+x}{1-x}\right)^r.\]
Substituting $2x$ in place of $x$ produces the result.
\end{proof}

The asymptotics of the type B $r$-derangements,  when $r$ is fixed and $n$ tends to infinity, are presented 
next. Standard methods of the analysis of the principal part of the generating function are used. Details 
appeared  in \cite[Theorem 5.2.1]{wilf-1990a}.

\begin{corollary}For any fixed $r\ge0$ and large $n$
\[\frac{d_{r,n}^{B}}{n!}\sim\frac{(-2)^n}{\sqrt e}\sum_{i=0}^r\binom{r}{i}2^{i}\left[\binom{-i-1}{n}-\frac{2i-1}{2}\binom{-i}{n}\right].\]
\end{corollary}

\section{A generalization of the $r$-Stirling numbers}
\label{generalization}

Let $\sigma = c_1c_2\cdots c_s$ be the usual cycle representation of a permutation into 
 $s$ disjoint cycles. The orders of the cycles are denoted by  $ord(c_{j}), \, 1 \leq j \leq s$.  For
  example, $\overline{4}\,3\,\overline{2}\,\overline{7}\,5\,\overline{8}\,1 \,9\,6$ is written as $(1\overline{4}\,\overline{7})(\overline{2}3)(5)(6\overline{8}9)$. Then
   $\,\,\, ord(c_{1})=3, \,\,\, ord(c_{2})=2,  \,\,\,\,ord(c_{3})=1, \,\,\, ord(c_{4})= 3$. The 
  derangements of type $B$ can be characterized by 
$$\D_n^B = \{\sigma =c_1\cdots c_s:\text{for all $i\in [s]$ either $ord(c_{i}) \geq2$ or 
$c_i\subseteq [\bar n]=\{\bar 1,\bar 2, \dots ,\bar n\}$}\}.$$
This leads to the following

\begin{definition}
Let $\sigma=c_1\cdots c_s$ be a permutation of type $B$ of length $n$ with $s$ cycles. The permutation
 $\sigma$ is a \emph{$m$-restricted permutation of type $B$} if for all $1\leq i\leq s$  either  the order of 
 each cycle satisfies $ord(c_{i}) \leq m$ or each item in $c_i$ is  signed, that is $c_i\subseteq [\bar n]$. The 
 permutation $\sigma$ is an \emph{$m$-associated permutation of type $B$} if for all $1\leq i\leq s$  
 either  the order of each cycle satisfies that $ord(c_{i}) \geq m$ or $c_i\subseteq [\bar n]$.
\end{definition}

Denote by $\mathcal{A}_{n,\leq m}^B (\mathcal{A}_{n,\geq m}^B)$ the set of all $m$-restricted 
(associated) permutations of type $B$ in $\D_n^B$. The \emph{$m$-restricted (associated)
 factorial numbers of type $B$}, $ A^B_{n,\leq m}$ is the corresponding cardinalities.
Note that $\D ^B_n = \mathcal{A}^B_{n,\geq 2}$. The \emph{associated (restricted) Stirling numbers of the first kind of type $B$}, ${n \brack k}_{\ge m}^B ({n \brack k}_{\leq m}^B$ , is the number of permutations of 
$A_{n,\geq (\leq) m}^B$   with $k$ cycles.  The basic relations
\begin{align*} 
A^B_{n,\geq m}:= \sum_{k=0}^n{n \brack k}_{\ge m}^B, \quad  \quad A^B_{n,\leq m}:= \sum_{k=0}^n{n \brack k}_{\leq m}^B,
\end{align*}
\noindent
hold.

The \emph{restricted Stirling numbers of the first kind (of type $A$)}  ${n \brack k}_{\leq m}$ enumerate the number of permutations on  $n$ elements with $k$ cycles with the restriction that none of the cycles  
contain more than $m$ items. Similarly, the \emph{associated Stirling numbers of the first kind} 
${n \brack k}_{\geq m}$ equals the  number  that each cycle  contains at  most $m$
 items.  Komatsu et al. \cite{komatsu-2016c} presented 
a variety of combinatorial properties for these sequences. Recently, Moll et al. \cite{moll-2018a} obtained 
 new combinatorial and arithmetical properties for them.

The restricted and associated  Stirling numbers of the first kind satisfy the recurrence relations
\begin{align}
{n \brack k}_{\leq m}&= \sum_{i=0}^{m-1} \frac{(n-1)!}{(n-1-i)!} {n-i-1 \brack k-1}_{\leq m},\label{rec4a}\\
{n \brack k}_{\geq m}&= \sum_{i=m-1}^{n-1} \frac{(n-1)!}{(n-1-i)!} {n-i-1 \brack k-1}_{\geq m}. \label{rec4b}
\end{align}
with  initial values
  \begin{align*}
{0 \brack 0}_{\leq m}=1   \quad  \text{and} \quad    {n \brack 0}_{\leq m}={0 \brack n}_{\leq m}=0, \\
{0 \brack 0}_{\geq m}=1   \quad  \text{and} \quad    {n \brack 0}_{\geq m}= {0 \brack n}_{\geq m}=0,
 \end{align*}
for $ n>0$. 
Introduce now the \textit{incomplete factorial numbers} $A_{i,\leq m}$ and $A_{i,\geq m}$, as 
the total number of incomplete permutations, that is 
\begin{align*} 
A_{n,\leq m}:= \sum_{k=0}^n{n \brack k}_{\le m} \quad A_{n,\geq m}:= \sum_{k=0}^n{n \brack k}_{\geq m}.
\end{align*}

\begin{theorem}
For $n\geq 0$, the identities
\begin{align*}
A_{n,\leq m}^B&=\sum _{i = 0}^n\binom{n}{i}2^iA_{i,\leq m}A_{n-i,\geq m+1},\\
A_{n,\geq m}^B&=\sum _{i=0}^n\binom{n}{i}2^iA_{i,\geq m}A_{n-i,\leq m-1},
\end{align*}
\noindent
hold.
\end{theorem}
\begin{proof}
Let $\sigma \in \mathcal{A}^B_{n,\geq k}$ and consider
 $A_1 = \{s\in [n]: \text{there is} \,\,  i\in [k]\text{ such that }\sigma ^i_s=s\},$ and $A_2= A_1^c$. It 
  is clear that $A_2\subseteq [n]$ and so there is no coloring on them.  Consider now the  the function 
$$\varphi : \mathcal{A}^B_{n,\leq k}\longrightarrow \bigcup _{i = 0}^n \binom{[n]}{n-i}\times [2]^{[i]}\times \mathcal{A}_{i,\leq k}\times \mathcal{A}_{n-1,\geq k+1},$$
given by $\varphi (\sigma)=(A_2,f,\left.\sigma \right|_{A_1},\left.\sigma \right|_{A_2} ),$ where $f(x)=\chi _{[\bar n]}(x).$  It is easy to check that this is a bijection, completing the proof.
\end{proof}

\subsection{The $r$-version}

This subsection consider the analog of the results presented above for the case of $r$-permutations. 

\begin{definition}
Let $n, \, r \in \mathbb{N}$ and $m\geq 1$. A \textit{type} B $\,\, m$-\textit{associated $r$-permutation} on the set $[n+r]$ is a
signed permutation on $[n+r]$,  with the  restriction that  the order  of each cycle  is at least $m$, and  the first $r$ elements (called \textit{special}) are restricted to be in distinct cycles.  The set of all $m$-associated $r$-permutation  of type $B$ on $[n+r]$ is denoted by $\mathcal{A}_{n,\geq m, r}^{B}$ and its cardinality by $A_{n,\geq m,r}^{B}$. \end{definition}

The number of elements of
$\mathcal{A}_{n,\geq m, r}^{B}$ with $k+r$ cycles is called the $m$-\textit{associated $r$-Stirling number of type} $B$  and
is denoted by
$\begin{displaystyle} \begin{bmatrix} n \\ k \end{bmatrix}_{\geq m, r}^{B} \end{displaystyle}$.
Counting over all possible cycles gives the relation
\begin{equation}
A_{n,\geq m, r}^{B} = \sum_{k=0}^{n} \begin{bmatrix} n \\ k \end{bmatrix}_{\geq m, r}^{B}.
\end{equation}

In the cases $m=0, \,1$, there are no restrictions on the sign, so one can color any element in
 $2^{n+r}$ ways and then choose the blocks as the normal  $r-$Stirling numbers of the first kind. It 
 follows that
         $${n\brack 0}_{\geq m,r}^B=2^{n+r}{n\brack 0}_r=2^{n+r}n!\binom{n+r-1}{r-1}.$$

The case $m=2$ was described  in Theorem \ref{rec-short}.          
The next result gives a general recurrence relation. This can be used to analyze the situation for $m>2$. 

The symbol $Par_{\leq c}(a,b)$ denotes the number of composition of $a$ into $b$ positive 
parts, each part of 
size at most $c$, then 
$$Par_{\leq c}(a,b)=\begin{cases} 
0, & c\leq 0\text{ and }(a\neq 0\text{ or }b\neq 0); \\
      \displaystyle \sum _{i=0}^b(-1)^i\binom{b}{i}\binom{a-ci-1}{b-1}, & \text{otherwise}. 
   \end{cases}
$$
Introduce the function 
\begin{equation}
\tau_{m,n}(j) = \begin{cases}
2^{j+1} & \quad \text{if} \,\,\, m-1 \leq j \leq n \\
1 & \quad \text{otherwise}.
\end{cases}
\end{equation}

\begin{theorem}\label{recgeneral}
For $n\geq 0$, $k, r \geq 1$ and $m>2$, the recursion 
   \begin{equation*}
  {n+1\brack k }_{\geq m,r}^B = \sum _{j=0}^{n} j! \tau_{m,n}(j)  \binom{n}{j}{n-j \brack k-1}_{\geq m,r}^B 
  +r  \sum _{j=0}^{n} (j+1)! \tau_{m,n+1}(j+1)  \binom{n}{j}{n-j\brack k}^B_{\geq m,r-1},
  \end{equation*}
\noindent
holds. The initial conditions are given by 
         \begin{align*}
{n \brack 0}_{\geq m ,r}^B&=n!\sum _{p=0}^r\sum _{j=0}^p\binom{r}{p}\binom{p}{j}\sum _{k=0}^n2^{n+p-k-j}Par_{\leq m-2}(k,j)Par_{\geq m-1}(n-k,p-j). 
\end{align*}
\end{theorem}
\begin{proof}
The initial conditions are discussed first. For $m>2$, choose $r-p$ elements from the $r$ 
special elements and assign 
them to its own bar. This can be done in $\binom{r}{p}$ ways. From the 
remaining
$p$ special elements, choose $j$  of them which will contain cycles of size less than $m$ so the have to have all sign. This cane done in $\binom{p}{j}$ ways. Then choose $k$ elements out of 
$n$ to put in these $j$ cycles in  $\binom{n}{k}Par_{\leq m-2}(k,j)k!$ ways. The remaining $n-k$ elements 
will be in the $p-j$ cycles with length greater or equal to $m$, in  $Par_{\geq m-1}(n-k,p-j)$ ways. Note 
that $$Par_{\geq c}(a,b)=Par_{\geq 1}(a-(c-1)b,b)=\binom{a-(c-1)b-1}{b-1}. $$ The 
inclusion-exclusion principle gives the expression for 
$Par_{\leq c}(a,b)$. Adding over all possibilities for  $p,j,k$ gives the result.\\

The recursion is discussed next. The discussion is divided into cases:
\begin{enumerate}
    \item Either $n+1$ is in a cycle without special elements and of length $<m$.
    This can be done by selecting the $j$ elements which are in these cycles and taking care of the cyclic order, to produce  $$\binom{n}{j}j!{n-j\brack k-1}^B_{\geq m,r}.$$ 
    \item $n+1$ is in a cycle without special elements of length $\geq m$. As before, select  the $j$ elements 
    and choose their signs in $2^{j+1}$ ways.
    \item $n+1$ is in a cycle with a special element and of length $<m$. Select  $j$ elements and take care 
    of  the cyclic order to obtain
    $$r\binom{n}{j}(j+1)!{n-j\brack k}_{\geq m,r-1}.$$
    \item $n+1$ is in a cycle with a special element and of length $\geq m.$ Select  $j$ elements as before 
    and now colored them in $2^{j+2}$ ways.
    \end{enumerate}
    Summing over these options gives the result.
\end{proof}

\subsection{General case of Howard's identity of type $B$}

Howard \cite{howard-1982a} established several combinatorial identities involving binomial coefficients
 and Stirling number of both kinds. Caicedo et al. \cite{caicedo-2019a} gave combinatorial proofs and generalization for some of them. For example, the  identity
$${n \brack n-k} = \sum_{\ell=0}^{k}\binom{n}{2k-\ell}{2k - \ell \brack k- \ell}_{\geq 2}$$
\noindent
is such an example. The next statement provides a $B$-analogue.

\begin{theorem}
The identity 
     \begin{align*}
       {n\brack k}^B_{\geq m,r}&=\sum _{p=0}^r\sum _{l = 0 }^k \binom{n}{ml}\binom{r}{p}\binom{n-ml}{(m-1)p}\frac{(2^m-1)^{l+p}(ml)!((m-1)p)!}{m^ll!}{n-ml-(m-1)p\brack k-l}^B_{\geq m+1,r-p}  \\
       &=n!\sum _{p=0}^r\sum _{l = 0 }^k\binom{r}{p}\frac{(2^m-1)^{l+p}}{m^ll!(n-m(l+p)+p)!}{n-m(l+p)+p\brack k-l}^B_{\geq m+1,r-p}  \\
     \end{align*}
 \noindent
 holds. 
\end{theorem}
\begin{proof}
Consider a type B $r$-permutation $\sigma \in \mathcal{A}^{B}_{n,\geq m,r}$ with $k+r$
 cycles and let $l$ be the number of cycles $\mathfrak{b}$ of size $m$ such that 
 $\mathfrak{b} \cap ([r]\cup [\bar r]) =\emptyset$ and $| \mathfrak{b}\cap [\overline{n+r}]|<m .$ Let 
 $p$ be the number of cycles with the above property with 
  $\mathfrak{b}\cap ([r]\cup [\bar r]) \neq \emptyset$.  The number of permutations with this two statistics are counted by
   the expression on the left-hand side. Adding for all cases of $p$ and $l,$ yields 
   the result.
\end{proof}

The special case  $r=0$ is stated next. 

\begin{corollary}
The identity 
     $${n \brack k}_{\geq m}^B=\sum _{l = 0}^k\binom{n}{ml}\frac{(ml)!}{m^ll!}(2^m-1)^l{n-ml\brack k-l}^B_{\geq m+1}$$
     holds. 
\end{corollary}

The case $m=1$ is similar. 

\begin{corollary}\label{Howard1}
For $n, k\geq 0$,  the identity
     \begin{align*}
       {n\brack k}^B_{r}&=\sum _{p=0}^r\sum _{\ell = 0 }^k\binom{r}{p}\binom{n}{\ell} {n-\ell \brack k-\ell}_{\geq 2, r-p},
     \end{align*}
holds. Moreover, if $r=0$ one obtains
          $${n \brack k}^B=\sum _{\ell=0}^k\binom{n}{\ell}{n-\ell \brack k-\ell}^B_{\geq 2}.$$
\end{corollary}

\subsection{The Riordan matrices}
It turns out that the sequence ${n \brack k}_{\ge m, r}^B$ can be encoded by a Riordan matrix. 

\begin{theorem}\label{teorio}
The matrix $\mathcal{C}_{\geq m, r}:=\left({n \brack k}_{\ge m, r}^B\right)_{n, k \geq 0}$ is an exponential Riordan array given by
$$\mathcal{C}_{\geq m, r}=\left(\left(\frac{1-x^{m-1}}{1-x}-\frac{2^mx^{m-1}}{1-2x}\right)^r, -\ln(1-2x)-\sum_{\ell=1}^{m-1}\frac{2^k-1}{k}x^k\right).$$
\end{theorem}

Theorem \ref{diagonals}  can be generalized as follows. The proof is left to the reader. 

\begin{theorem} If $m>0,$ the two diagonals below the main diagonal are 
\begin{align*}
    {n+1\brack n}^B_{\geq m,r}&=2^{(n+r+1)\delta _{m,1}+2\delta _{m,2}-1}(n+1)(n+2r),\\
    {n+2\brack n}^B_{\geq m, r}&=\frac{2^{(n+r+2)\delta _{m,1}}}{12}\binom{n+2}{2}\left (3 \cdot 
    2^{4\delta _{m,2}}(4r(r+n-1)+n(n-1))+2^{3(\delta _{m,2}+\delta _{m,3})+3}(n+3r)\right )
\end{align*}
where $\delta _{a,b}$ is the Kronecker delta function.
\end{theorem}

\medskip

\noindent
{\bf Acknowledgements}. 

The first author thanks for hospitality the  Department of Mathematics of Universidad Nacional de Colombia, Bogot\'a, Colombia, where the presented work was initiated. The third author thanks the Department of Mathematics of the Tulane University, New Orleans.   The last author is a graduate 
student at Tulane 
University.  The research of the third author was partially supported by the Universidad Nacional de Colombia, Project No. 46240.

\bibliography{/Users/vmh/Dropbox/AllRef/official1}
\bibliographystyle{plain}
\end{document}

\end{document}